\documentclass[a4paper, 12pt]{article}
\usepackage{}

\usepackage{mathrsfs}
\usepackage{epsfig}
\usepackage{amsmath}
\usepackage{amssymb}
\usepackage{latexsym}
\usepackage{amsfonts}
\usepackage{amsthm}

\usepackage[numbers,sort&compress]{natbib}
\usepackage{graphicx}
\ifx\pdfoutput\undefined  \else
 \fi

\usepackage[centerlast]{caption2}

\usepackage{color}

\usepackage{txfonts}
\usepackage{amssymb}
\usepackage{mathrsfs}
\usepackage{amsfonts}

\newtheorem{theorem}{Theorem}[section]
\newtheorem{lemma}[theorem]{Lemma}
\newtheorem{cor}[theorem]{Corollary}
\newtheorem{prop}[theorem]{Proposition}
\newtheorem{conj}[theorem]{Conjecture}
\newtheorem{remark}[theorem]{Remark}
\newtheorem{exam}[theorem]{Example}
\newtheorem{definition}[theorem]{Definition}

\usepackage{latexsym,amssymb}
\usepackage{amsmath}

\newcommand{\F}{{\mathbb F}}

\begin{document}

\title{{Semigroup rings and algebraically independent sequences with respect to idempotents in commutative semigroups}}
\author{
Guoqing Wang\\
\small{School of
Mathematical Sciences, Tiangong University, Tianjin, 300387, P. R. China}\\
\small{Email: gqwang1979@aliyun.com}
\\
}
\date{}
\maketitle

\begin{abstract} For any finite abelian group $G$ and commutative unitary ring $R$, by $R[G]$ we denote the group algebra over $R$. Let $T=(g_1,\ldots,g_{\ell})$ be a sequence over the group $G$. We say $T$ is algebraically zero-sum free over R if $\prod\limits_{i=1}^r(X^{g_i}-a_i)\neq 0\in R[G]$ for all $a_1,\ldots,a_{\ell}\in R\setminus \{0\}$. Let $d(G,R)={\rm sup}\{|T|: T  \mbox{ is an algebraically zero-sum free sequence over } R \mbox{ of terms from }G\}.$ This invariant of the group algebra $R[G]$ plays a powerful role in the research for the zero-sum theory. In this paper, we generalize this invariant to the semigroup algebra $R[S]$ for a commutative periodic semigroup $S$. We give the best possible lower and upper bounds for $d(S,R)$ for a general commutative periodic semigroup $S$. In case that $K$ is a field, and $S$ is a finite commutative semigroup, we give more precise result, including the equality for Clifford semigroups, Archimedean semigroups and elementary semigroups, which covers all types of irreducible components associated with the semilattice decomposition and the subdirect product decomposition of a commutative semigroup.  Also, the invariant $d(S,K)$ was applied to the study of some zero-sum invariants in semigroups.  One conjecture on the equality for $d(S,K)$ in case $K$ is an algebraically closed field of characteristic zero was proposed which has been also partially affirmed in this paper.
\end{abstract}


\section {Introduction}

Let $G$ be a finite abelian group written additively. The Davenport constant ${\rm D}(G)$ of $G$ is defined as
the smallest positive integer $\ell$ such that, every sequence of $\ell$ terms from $G$
contains some terms with sum being the identity element. This invariant was first formulated by K. Rogers \cite{Rogers}, and popularized by H. Davenport in the 1960's, notably for its link with algebraic number theory (as
reported in \cite{Olson1}), and has been investigated extensively in the past 60 years. This combinatorial
invariant was found with applications in other areas, including Factorization Theory of Algebra (see
\cite{GH,GRuzsa}), Classical Number Theory, Graph Theory, and Coding Theory.
Although the constant has been proposed for many years, and it is generally considered very important, only very few direct progress on its value has been made. The following two theorems are the landmark achievements in the investigation on the Davenport constant of finite abelian groups.

\noindent\textbf{Theorem A}. \cite{EmdeBoas} Let $G$ be a finite abelian group, and $\exp(G)$ the exponent of $G$. Then ${\rm D}(G)\leq \exp(G)+ \exp(G) {\rm ln}\frac{|G|}{\exp(G)}$, in particular, if $G$ is cyclic, then ${\rm D}(G)= \exp(G)+ \exp(G) {\rm ln}\frac{|G|}{\exp(G)}=|G|$.

\noindent\textbf{Theorem B}. \cite{Olson2} Let $G$ be a finite abelian $p$-group for some prime $p$. Say $G\cong \mathbb{Z}_{p^{a_1}}\oplus\cdots\oplus \mathbb{Z}_{p^{a_r}}$ with $1\leq n_1\leq \cdots\leq n_r$. Then ${\rm D}(G)=d^*(G)+1$ where $d^*(G)$ is given as \eqref{definition of d*(G)}.

As a consequence of the Fundamental Theorem for finite abelian groups, any nontrivial finite abelian group can be written uniquely as the direct sum  $\mathbb{Z}_{n_1}\oplus\cdots\oplus \mathbb{Z}_{n_r}$ of cyclic groups  $\mathbb{Z}_{n_1},\ldots, \mathbb{Z}_{n_r}$ with $1 < n_1 \mid \cdots \mid n_r$. Then we denote
\begin{equation}\label{definition of d*(G)}
{\rm d}^*(G) = \sum\limits_{i=1}^r (n_i-1).
\end{equation}

Theorem A is was applied by W.R. Alford, A. Granville and C. Pomerance \cite{AGP} to prove that there are
infinitely many Carmichael numbers which solve the old problem starting form 1912.
Theorem B for some sense is the unique {\sl general} type of finite abelian groups, for which the precise value is known. For more result on the precise value of Davenport constant in special types, one is referred to \cite{GaoGeroldingersurvey}.

One thing worth mentioning is that Group algebra is the crucial tool to prove both Theorems A and B (see Theorem C below), for which the key is the invariant $d(G,R)$ given in Definition \ref{Definition d(s,R) for groups}. Here, we first review some basic definition for (\sl commutative) group ring. Let $G$ be a finite abelian group, and let $R$ be a commutative unitary ring. The group ring $R[X;G]$, of the group $G$ over the ring $R$, is a free-$R$ module with basis $\{X^g: g\in G\}$, for which the multiplication is defined by
$$(\sum\limits_{g\in G} a_g X^g) (\sum\limits_{g\in G} b_g X^g)=\sum\limits_{g\in G} (\sum\limits_{h\in G}a_h b_{g-h}) X^g.$$

\begin{definition}\label{Definition d(s,R) for groups} \cite{GeGyZhong} \ Let $G$ be a finite abelian group, and $R$ a commutative unitary ring. Given a sequence $T=g_1\cdot\ldots\cdot g_{\ell}\in \mathcal{F}(G)$, we say the sequence $T$ is algebraically zero-sum free (over R) provided that
$(a_1-X^{g_1})\cdot\ldots\cdot (a_{\ell}-X^{g_{\ell}})\neq 0\in R[X;G]$ for all $a_1,\ldots,a_{\ell}\in R\setminus \{0\}$.
Moreover, we define ${\rm d}(G,R)={\rm sup}\{|T|: T \mbox{ is a sequence over } G \mbox{ and algebraically zero-sum free over } R\}.$
\end{definition}

By using the character theory of finite group and combinatorial methods, the following result was obtained on ${\rm d}(G,R)$, which gives Theorems A and B as stated before.

\noindent \textbf{Theorem C}. (see Theorem 2.2.6 of \cite{GRuzsa}, Theorem 3.4.16 of \cite{GeGyZhong}) Let $G$ be a finite abelian group with $n=\exp(G)$, and let $R$ be domain.
Then the following conclusions hold:

(i) ${\rm D}(G)\leq {\rm d}(G,R)+1$;

(ii) Suppose $R$ is a splitting field of $G$ (that is, $R$ is a field with $\sharp\{\zeta\in R: \zeta^n=1_R\}=n$). Then ${\rm d}(G,R)\leq (n-1)+n\log\frac{|G|}{n}$ and equality holds if $G$ is cyclic, where ${\rm log}$ denotes the natural logarithm.

(iii) Suppose $n=p^{t}$ and $R=\mathbb{F}_p$ where $p$ is prime and $t>0$. Then ${\rm d}(G,R)=d^*(G)$.

Group algebra $R[X;G]$ over suitable commutative ring $R$ have turned out to be powerful tools for a growing variety of questions from combinatorics and number theory. It is important and meaningful to develop this method from the combinatorial perspective (see \cite{EmdeBoas,Gaojcta,GaoGeroldingernumberofpgroup,GaoGeroldingerRocket,GaoLiArs,GaoDWangIsr,GeGyZhong,GeroldingerScheiderDiscre,Grynkiewiczmono,Olson1,Smertnig} for example).
On the other hand, the zero-sum type of problems have been also investigated recently
in the setting of semigroups (see \cite{Nkravitz}, \cite{wangDavenportII}, \cite{wangAddtiveirreducible}, \cite{wangidempotent}, \cite{wangCommunication}, \cite{wang-zhang-wang-qu}, \cite{wang-zhang-qu} for example).
In this paper, we generalize the invariant $d(G;R)$ to commutative semigroups (see Definition \ref{Definition d(S,R)}). Among other results, we give the best possible lower and upper bounds of $d(S;R)$ for a finite commutative semigroup $S$ and commutative unitary ring $R$ via to the Sch${\rm \ddot{u}}$tzenberger groups and the maximal length of principal ideals chain of $S$, and we determine $d(S;R)$ in case that $S$ or a clifford semigroup, or an archimedean semigroup, or an elementary semigroup.
Moreover, we try to connect $d(S;R)$ with one combinatorial invariant, Erd\H{o}s-Burgess constant of a finite commutative semigroup $S$, and get some preliminary results.

\section{Notation and terminology}

For integers $a,b\in \mathbb{Z}$, we set $[a,b]=\{x\in \mathbb{Z}: a\leq x\leq b\}$. For a real number $x$, we denote by $\lceil x\rceil$ the smallest integer that is greater than or equal to $x$.

Let $\mathcal{S}$ be a commutative semigroup written additively, where the operation of $\mathcal{S}$ is denoted as $+$.
For any positive integer $m$ and any element $a\in \mathcal{S}$, we denote by $ma$ the sum $\underbrace{a+\cdots+a}\limits_{m}$. An element $e$ of $\mathcal{S}$ is said to be idempotent if $e+ e=e$. A cyclic semigroup is a semigroup generated by a single element $x$, denoted $\langle x\rangle$, consisting of all elements which can be represented as $m x$ for some positive integer $m$.
If the cyclic semigroup $\langle x\rangle$ is infinite then $\langle x\rangle$ is isomorphic to the semigroup of $\mathbb{N}$ with addition (see \cite{Grillet monograph}, Proposition 5.8), and if $\langle x\rangle$ is finite
then the least integer $k>0$ such that $kx=tx$ for some positive integer $t\neq k$ is called the {\sl index} of $x$,  then the least integer $n>0$ such that $(k+n)x=k x$ is called the {\sl period} of $x$.
We denote a finite cyclic semigroup of index $k$ and period $n$ by $C_{k; n}$, denoted as $\mathcal{P}(x)$.
By $\rho(x)$ we denote the least positive integer $m$ such that $m x=e$ is the unique idempotent in the cyclic semigroup $\langle x\rangle$.  Precisely,
\begin{equation}\label{equation rho(x)=knn}
\rho(x)=\left\lceil\frac{k}{n}\right\rceil n
 \end{equation} can be seen from Lemma \ref{Lemma cyclic semigroup} below. By $e(x)$ we denote the unique idempotent in the cyclic semigroup $\langle x\rangle$ generated by $x$, i.e.,
\begin{equation}\label{equation e(x)=rho(x)x}
e(x)=\rho(x) x.
\end{equation}
Then $\rho(x)=1$ if and only if $x$ is an idempotent. One thing worth remarking is that somewhere in this manuscript, we shall write $e(x)=x+(2\rho(x)-1)x$ instead of $x+(\rho(x)-1)x$ in the process of calculations to avoid the confusion in case that $x$ is an idempotent and $\rho(x)-1=0$.
For a finite commutative semigroup $S$, we denote $${\rm exp}(S)={\rm lcm}\{\mbox{ period of } x: x\in S\},$$
We say one element $b$ is {\sl periodic} if $\langle b\rangle$ is finite. If every element of the semigroup $S$ is periodic, then we call the semigroup {\bf periodic}. In particular, a finite semigroup is periodic.

\noindent $\bullet$ {\sl Note that if $k=1$ the semigroup $C_{k; n}$ reduces to a cyclic group of order $n$ which is isomorphic to the additive group $\mathbb{Z}_{n}$ of integers modulo $n$; If the semigroup $S$ is a torsion group, then  $\rho(x)$ reduces to the order $o(x)$ of $x$.}

For any element $a\in \mathcal{S}$, let
$$(a)=\{a\}\cup \{a+c: c\in \mathcal{S}\}$$ denotes the {\sl principal ideal} generated by the element $a\in \mathcal{S}$.
The {\bf Green's preorder} on the semigroup $\mathcal{S}$, denoted $\preceq_{\mathcal{H}}$, is defined by
$$a \preceq_{\mathcal{H}} b\Leftrightarrow a=b \ \ \mbox{or}\ \ a=b+c \mbox{ for some } c\in \mathcal{S},$$  equivalently, $$(a)\subseteq (b).$$ Green's congruence, denoted
$\mathcal{H}$, is a basic relation introduced by Green for semigroups which is defined by:
$$a \ \mathcal{H} \ b \Leftrightarrow a \ \preceq_{\mathcal{H}} \ b \mbox{ and } b \ \preceq_{\mathcal{H}} \ a \Leftrightarrow (a)=(b).$$ By $S\diagup\mathcal{H}$ we denotes the quotient semigroup of $S$ modulo the congruence $\mathcal{H}$. For any element $a$ of $\mathcal{S}$,  let $H_a$ be the congruence class by $\mathcal{H}$ containing $a$, i.e., $H_a=\{x\in S: x \mathcal{H} a\}$.  We write $a\prec_{\mathcal{H}} b$ to mean that $a \preceq_{\mathcal{H}} b$ but $H_a\neq H_b$.
For any element $a\in S$, we define $\psi(a)$ to be the largest length $\ell\in \mathbb{N}$ of strictly ascending principal ideals chain of $\mathcal{S}$ starting from $(a)$, i.e., the largest $\ell\in \mathbb{N}$ such that there exist $\ell$ elements $a_1,\ldots,a_{\ell}\in \mathcal{S}$ with $$(a) \subsetneq  (a_1) \subsetneq \cdots \subsetneq (a_{\ell}),$$ or equivalently,
$$a \prec_{\mathcal{H}} a_1\prec_{\mathcal{H}}\cdots \prec_{\mathcal{H}} a_{\ell}.$$
We define $\Psi(S)$ to be the largest length $\ell\in \mathbb{N}$ of strictly ascending principal ideals chain of $\mathcal{S}$. Then $$\Psi(S)={\rm sup}\{\psi(a): a\in S\}.$$

For any subset $A\subseteq \mathcal{S}$, let $${\rm St}(A)=\{c\in \mathcal{S}: c+a\in A \ \mbox{for every} \ a\in A\}$$ be the stabilizer of the set $A$ in the semigroup $\mathcal{S}$, which is a subsemigroup of $\mathcal{S}$. Then we introduce the definition of {\bf Sch${\rm \ddot{u}}$tzenberger group} which palys a key role in giving the bounds for $d(S,R)$.

Each $c\in {\rm St}(H_a)$ induces a mapping $$\gamma_c:H_a\rightarrow H_a$$ defined by
\begin{equation}\label{equation gammac}
\gamma_c: x\mapsto c+x
 \end{equation}for every $x\in H_a$, which we write as a operator $$\gamma_c \circ x=c+x.$$ Let $$\Gamma(H_a)=\{\gamma_c: c\in {\rm St}(H_a)\}.$$
It is known that $\Gamma(H_a)$ is an abelian group with the operation
\begin{equation}\label{equation operation}
\gamma_c+\gamma_d=\gamma_{c+d},
\end{equation} which is discover by M.P. Sch${\rm \ddot{u}}$tzenberger  in 1957 (see Section 4 of Chapter I in \cite{Grillet monograph}), and is called the Sch${\rm \ddot{u}}$tzenberger group of $H_a$. By \eqref{equation operation}, we see that there exists a semigroup homomorphism $\pi_a$ of ${\rm St}(H_a)$ onto $\Gamma(H_a)$, defined by
\begin{equation}\label{equation homo to shuzen}
\pi_a: c\mapsto \gamma_c
\end{equation}
for every $c\in {\rm St}(H_a)$.
A commutative semigroup is {\bf archimedean} provided that for every $a,b\in S$, there exist some $n>0$ and $t\in S$ such that $na=t+b$. A semigroup $S$ is called semilattice if all its elements are idempotent. Let $\varphi$ be a homomorphism of a semigroup $S$ into a semilattice $Y$ (the operation of $Y$ is denoted as $\wedge$). Then $S_a=\{x\in S: \varphi(x)=a\}$ is a subsemigroup of $S$
for $a\in Y$, and $S_aS_b=\{xy:x\in S_a, y\in S_b\}\subset S_{a\wedge b}$ for all $a,b\in Y$. That is, $S$ is a semilattice of semigroups when there exists a semilattice $Y$ and a partition $S=\bigcup\limits_{a\in Y} S_a$ of $S$ into subsemigroups $S_a$ (one for every $a\in Y$) such that $S_aS_b=\{xy:x\in S_a, y\in S_b\}\subset S_{a\wedge b}$ for all $a,b\in Y$. By the structure theory for semigroups, we know that every commutative semigroup always has a semilattice decomposition and every component is an archimedean semigroup (see Chapter III in \cite{Grillet monograph}). A commutative {\bf Clifford} semigroup is a semigroup with a semilattice decomposition $S=\bigcup G_a$ in which every semigroup $G_a$ is a group. Clifford semigroups are also called a commutative inverse semigroups or commutative regular semigroups. Without mentioning the semilattice decomposition, one can equivalently define Clifford semigroups as follows:  A semigroup $S$ is called regular provided that for any element $a\in S$, there exists $a'\in S$ such that $a'aa'=a$ and $aa'a=a$. A commutative regular semigroup is called a clifford semigroup. Clifford semigroups have been studied extensively and have provided a testing ground for many ideas and problems (see \cite{Grillet monograph}).
An {\bf elementary semigroup} $S=G\cup N$ is an ideal extension of a nilsemigroup $N$ by a group with zero $G\cup \{\infty\}$. In this case, the group $G$ acts on the set $N$ by left multiplication ($g\cdot x=gx$). By Birkhoff Theorem (see \cite{Grillet monograph} Chapter IV, Theorem 4.1),  every commutative semigroup is a subdirect product of subdirectly irreducible commutative semigroups. Also every finite subdirectly irreducible commutative semigroup is either a group or a nilsemigroup or elementary.

Now we shall introduce the definition for semigroup ring.
For any commutative unitary ring $R$ and any commutative semigroup $S$, let $R[X;S]$ be the semigroup ring, consisting of all functions $f$ from $S$ into $R$ that are finitely nonzero (i.e., $f(s)=0_R$ holds for almost all elements  of $S$ except for only finitely many), with addition and multiplication defined in $R[X;S]$ as follows:
$f+g(s)=f(s)+g(s)$ and $(fg)(s)=\sum\limits_{t+u=s} f(t)g(u)$ where the symbol $\sum\limits_{t+u=s}$ indicates that  the sum is taken over all pairs $(t,u)$ of elements of $S$ such that $t+u=s$, and it is understood that $(fg)(s)=0$ if $s$ is not expressible in the form $t+u$ for any $t,u\in S$. For the convenience, the element $f\in R[X;S]$ can be denoted to be $\sum\limits_{s\in S}f(s) X^{s}$ as given for the group algebra, and the operation is as $$(\sum\limits_{s\in S} a_s X^s) (\sum\limits_{s\in S} b_s X^s)=\sum\limits_{s\in S} (\sum\limits_{t+u=s}a_t b_u) X^s.$$
For more notation and definitions on semigroup rings, one is referred to \cite{Gilmer}.
Then we introduce the definition of $d(S,R)$ associated with semigroups rings $R[X;S]$.

\begin{definition}\label{Definition d(S,R)}  For a commutative periodic semigroup $S$ and a commutative unitary ring $R$, let ${\rm d}(S,R)$ denote the supremum of all $\ell\in \mathbb{N}\cup \{\infty\}$ having the following property:

There exists some sequence $T=s_1\cdot\ldots\cdot s_{\ell}$ of length $\ell$ over $S$ such that $$(X^{s_1}-a_1X^{e(s_1)})\cdot\ldots\cdot (X^{s_{\ell}}-a_{\ell} X^{e(s_{\ell})})\neq 0\in R[X;S]$$ for all $a_1,\ldots,a_{\ell}\in R\setminus \{0\}$, where $e(s_i)$ denotes the unique idempotent in the finite cyclic semigroup generated by $s_i$ ($i\in [1,\ell]$).
\end{definition}

\begin{remark} We remark that if the commutative semigroup $S$ is a finite abelian group, then $e(g_i)$ is equal to the unique identity of $G$ and the definition for ${\rm d}(S,R)$ reduces to one for groups given as Definition \ref{Definition d(s,R) for groups}. If the semigroup (group) $S$ is empty, we let $d(S;R)=0$ which will be used in the proof latter and cause no confusion.
\end{remark}

Let $R$ be a commutative ring. For any module $M$ over a commutative ring $A$. A ring-homomorphism $\varphi:R\rightarrow {\rm End}_A(M)$ is a representation of $R$ on $M$.

We also need notation and terminologies on sequences over semigroups and follow the notation of A. Geroldinger, D.J. Grynkiewicz and
others used for sequences over groups (cf. [\cite{Grynkiewiczmono}, Chapter 10] or [\cite{GH}, Chapter 5]). Let ${\cal F}(\mathcal{S})$
be the free commutative monoid, multiplicatively written, with basis
$\mathcal{S}$. We denote multiplication in $\mathcal{F}(\mathcal{S})$ by the boldsymbol $\cdot$ and we use brackets for all exponentiation in $\mathcal{F}(\mathcal{S})$. By $T\in {\cal F}(\mathcal{S})$, we mean $T$ is a sequence of terms from $\mathcal{S}$ which is
unordered, repetition of terms allowed. Say
$T=a_1a_2\cdot\ldots\cdot a_{\ell}$ where $a_i\in \mathcal{S}$ for $i\in [1,\ell]$.
The sequence $T$ can be also denoted as $T=\mathop{\bullet}\limits_{a\in \mathcal{S}}a^{[{\rm v}_a(T)]},$ where ${\rm v}_a(T)$ is a nonnegative integer and
means that the element $a$ occurs ${\rm v}_a(T)$ times in the sequence $T$. By $|T|$ we denote the length of the sequence, i.e., $|T|=\sum\limits_{a\in \mathcal{S}}{\rm v}_a(T)=\ell.$ By $\varepsilon$ we denote the
{\sl empty sequence} in $\mathcal{S}$ with $|\varepsilon|=0$. We call $T'$
a subsequence of $T$ if ${\rm v}_a(T')\leq {\rm v}_a(T)\ \ \mbox{for each element}\ \ a\in \mathcal{S},$ denoted by $T'\mid T,$ moreover, we write $T^{''}=T\cdot  T'^{[-1]}$ to mean the unique subsequence of $T$ with $T'\cdot T^{''}=T$.  We call $T'$ a {\sl proper} subsequence of $T$ provided that $T'\mid T$ and $T'\neq T$. In particular, the  empty sequence  $\varepsilon$ is a proper subsequence of every nonempty sequence.
Let $\sigma(T)=a_1+\cdots +a_{\ell}$ be the sum of all terms from $T$. The sequence $T$ is called an additively reducible (reducible) sequence if $T$ contains a proper subsequence $T'$ with $\sigma(T')=\sigma(T)$, and is called an  additively irreducible (irreducible) sequence if otherwise.
We call $T$ a {\sl zero-sum} sequence provided that $\mathcal{S}$ is a monoid and $\sigma(T)=0_{\mathcal{S}}$.
In particular,
if $\mathcal{S}$ is a monoid,  we allow $T=\varepsilon$ to be empty and adopt the convention
that $\sigma(\varepsilon)=0_\mathcal{S}.$

\noindent \textbf{Definition D.}  (see \cite{wangAddtiveirreducible}) \ {\sl  The Davenport constant of the commutative semigroup $\mathcal{S}$, denoted ${\rm D}(\mathcal{S})$, is defined to be the smallest $\ell\in \mathbb{N}\cup\{\infty\}$ such that every sequence $T\in \mathcal{F}(\mathcal{S})$ of length at least $\ell$ is reducible. For any element $a\in \mathcal{S}$, the relative Davenport constant of $\mathcal{S}$ with respect to $a$, denoted ${\rm D}_a(\mathcal{S})$, is defined to be the largest $\ell\in \mathbb{N}\cup \{\infty\}$ such that there exists an irreducible sequence $T\in \mathcal{F}(\mathcal{S})$ with $|T|=\ell$ and $\sigma(T)=a$. In particular, if $a=0_S$ is the identity of $S$, then ${\rm D}_a(\mathcal{S})=0$.}

We say  the sequence $T$ is an {\sl idempotent-sum sequence} if $\sigma(T)$ is an idempotent; and we say $T$
is an {\sl idempotent-sum free sequence} if $T$ contains no nonempty idempotent-sum subsequence.
It is worth remarking that when the commutative semigroup $\mathcal{S}$ is an abelian group, the notion {\sl zero-sum sequence} and {\sl idempotent-sum sequence} make no difference.

Let $K$ be a field.
For a finite commutative semigroup $S$ such that $\exp(S)=n$, we say $K$ is a splitting field of the semigroup $S$ provided that $K$ is a splitting field of the group $\mathbb{Z}_n$, i.e., $\sharp\{\zeta\in K: \zeta^n=1_K\}=n$.
In particular, an algebraically closed field of characteristic zero is always a splitting field of any finite commutative semigroup.

\section{d(S,R) with $R$ being a general commutative unitary ring}

In this section, we shall give the
sharp upper and lower bounds for $d(S,R)$ with $R$ being a general commutative unitary ring, and show that the upper bound is attained in case that $S$ is clifford or archimedean. The results are given as Theorem \ref{Theorem General bound for finite commutative semigroup}, Theorem \ref{Theorem for finite}, Corollary \ref{Corollary Clifford} and Corollary \ref{Corollary Archimedean}.

\begin{theorem}\label{Theorem General bound for finite commutative semigroup} Let $R$ be a commutative unitary ring, and
let $S$ be a commutative periodic semigroup such that ${\rm D}(S\diagup \mathcal{H})<\infty$. Then $d(S,R)<\infty$ if and only if $d(\Gamma(H_a),R)<\infty$ for each $a\in S$. Moreover, in case of  $d(S,R)<\infty$, we have
the following best possible upper and lower bounds $$\max\limits_{a\in S}\left\{d(\Gamma(H_a),R)\right\} \leq d(S,R)\leq\max\limits_{a\in S}\left\{\epsilon_a \ {\rm D_{\bar{a}}}(S\diagup \mathcal{H}) +d(\Gamma(H_a),R)\right\},$$ where $$\begin{array}{llll}\epsilon_a=\left\{\begin{array}{llll}
               0,  & \mbox{if \ \ } {\langle a\rangle} \mbox{ is a group };\\
               1,  &  otherwise,\
              \end{array}
              \right.
\end{array}$$
and  $\bar{a}$ denotes the image of $a$ under the canonical epimorphism of $S$ onto $S\diagup \mathcal{H}$.
\end{theorem}

We remark that the condition ${\rm D}(S\diagup \mathcal{H})<\infty$ is not a necessary condition such that $d(S,R)<\infty$ for a commutative periodic semigroup $S$, which can be seen in Example \ref{Example D not finite}.
Without revising the proof substantially, we can obtain the same conclusion as Theorem \ref{Theorem General bound for finite commutative semigroup} for a commutative periodic semigroup $S$ such that ${\rm sup}\{{\rm D_{\bar{a}}}(S\diagup \mathcal{H}: a\in S \mbox{ with } \epsilon_a=1\}<\infty$ (see Theorem \ref{Theorem rivising}).

If $S$ is a finite commutative semigroup, then $S\diagup \mathcal{H}$ is finite and ${\rm D}(S\diagup \mathcal{H})<\infty$ holds true (see Proposition D in \cite{wangAddtiveirreducible}). That is, in case that $S$ is finite,  then $d(S,R)<\infty$ if and only if $d(\Gamma(H_a),R)<\infty$ for each $a\in S$. Moreover, if we take $R$ to be an algebraically closed field of characteristic zero, then $d(\Gamma(H_a),R)<\infty$ always holds. Hence, we have the following theorem.

\begin{theorem}\label{Theorem for finite} Let $S$ be a finite commutative semigroup, and let $K$ be a field.
Suppose that either (i) $K$ is a splitting field of $S$, or (ii) ${\rm Char}(K)=p$ and $\exp(S)$ is a $p$-power for some prime $p$.
Then $d(S,K)<\infty$, and moreover, $$\max\limits_{a\in S}\left\{d(\Gamma(H_a),K)\right\} \leq d(S,K)\leq\max\limits_{a\in S}\left\{\epsilon_a \ {\rm D_{\bar{a}}}(S\diagup \mathcal{H}) +d(\Gamma(H_a),K)\right\}\leq \max\limits_{a\in S}\left\{\epsilon_a \ (\psi(a)+1)+d(\Gamma(H_a),K)\right\},$$ where $\epsilon_a$ and $\bar{a}$ are given in Theorem \ref{Theorem General bound for finite commutative semigroup}. In particular, the bounds are attained in the case that $S$ is a clifford semigroup or an archimedean semigroup.
\end{theorem}

In the following, we shall show that the bounds in Theorem \ref{Theorem for finite} are attained in finite commutative clifford semigroups, finite archimedean semigroups and finite elementary semigroups, which also show the sharpness of the bounds in Theorem \ref{Theorem General bound for finite commutative semigroup}.

\begin{cor}\label{Corollary Clifford} Let $S$ be a finite commutative clifford semigroup, and let $K$ be a field. Suppose that either (i) $K$ is a splitting field of $S$, or (ii) ${\rm Char}(K)=p$ and $\exp(S)$ is a $p$-power for some prime $p$. For each $a\in S$,  we have that $\epsilon_a=0$, and the Green congruence class $H_a$ is a maximal subgroup containing $a$ and $\Gamma(H_a)\cong H_a$. In particular, $$d(S,K)=\max\limits_{a\in S}\left\{d(\Gamma(H_a),K)\right\}=\max\limits_H \{d(H,K)\}$$ where $H$ takes all subgroup of $S$.
\end{cor}

\begin{cor}\label{Corollary Archimedean} Let $S$ be a finite commutative archimedean semigroup,  and let $K$ be a field. Suppose that either (i) $K$ is a splitting field of $S$, or (ii) ${\rm Char}(K)=p$ and $\exp(S)$ is a $p$-power for some prime $p$. Let $G_S$ be the largest subgroup of $S$. Then the following conclusions hold:

(i) $d(S,K)=\max\limits_{a\in S}\left\{\epsilon_a \ (\psi(a)+1)+d(\Gamma(H_a),K)\right\}=\max\left(\Psi(S),d(G_S,K)\right)$;

(ii) $d(S,K)\geq \max\left(\Psi(S),{\rm D}(G_S)-1\right)$, and the equality holds if $G_S$ is cyclic or (ii) holds.
\end{cor}

\begin{cor}\label{Corollary elementarysemigroup} Let $S$ be a finite commutative elementary semigroup,  and let $K$ be a field. Suppose that either (i) $K$ is a splitting field of $S$, or (ii) ${\rm Char}(K)=p$ and $\exp(S)$ is a $p$-power for some prime $p$.  Then
$$d(S,K)=\max\limits_{a\in S}\left\{\epsilon_a \ {\rm D_{\bar{a}}}(S\diagup \mathcal{H}) +d(\Gamma(H_a),K)\right\}=\max\limits_{a\in S}\left\{\epsilon_a \ (\psi(a)+1)+d(\Gamma(H_a),K)\right\}.$$
\end{cor}

\medskip

To prove the main results of this section, we need some lemmas.

\begin{lemma}\label{lemma with idempotent} \ Let $S$ be a periodic commutative semigroup. For any $s_1,\ldots,s_k\in S$, we have that $e(\sum\limits_{i=1}^k s_i)=\sum\limits_{i=1}^k e(s_i)$.
\end{lemma}

\begin{proof} It suffices to show the conclusion when $k=2$. Denote $m_i=\rho(s_i)$ for $i=1,2$. Then it follows from \eqref{equation e(x)=rho(x)x} that
$m_1m_2(s_1+s_2)=m_1m_2s_1+ m_1m_2 s_2=m_2(m_1 s_1)+m_1 (m_2 s_2)=m_2 e(s_1) +m_1 e(s_2)=e(s_1)+e(s_2)$. Since the sum of idempotents of $S$ is still an idempotent, we have that $e(s_1)+e(s_2)=m_1m_2(s_1+s_2)$ is an idempotent in the cyclic semigroup $\langle s_1+s_2\rangle$ generated by the element $s_1+s_2$. It follows that $e(s_1+s_2)=e(s_1)+e(s_2)$, completing the proof.
\end{proof}

\begin{lemma}\label{Lemma cyclic semigroup} (\cite{Grillet monograph},  Chapter I, Lemma 5.7, Proposition 5.8, Corollary 5.9) \ Let $\mathcal{S}=C_{k; n}$ be a finite cyclic semigroup generated by the element $x$. Then  $\mathcal{S}=\{x,\ldots,k x,(k+1)x,\ldots,(k+n-1)x\}$
with
$$\begin{array}{llll} & ix+jx=\left \{\begin{array}{llll}
               (i+j)x, & \mbox{ if } \  i+j \leq  k+n-1;\\
                tx, &  \mbox{ if }  \ i+j \geq k+n, \ \mbox{ where} \  k\leq t\leq k+n-1 \ \mbox{ and } \ t\equiv i+j\pmod{n}. \\
              \end{array}
           \right. \\
\end{array}$$
Moreover, there exists a unique idempotent, $\ell x$, in the cyclic semigroup $\langle x\rangle$, where $$\ell\in [k,k+n-1] \  \mbox{ and }\  \ell\equiv 0\pmod {n};$$
\end{lemma}

\begin{lemma}\label{Lemma subsemigroup} Let $S$ be a finite commutative semigroup, and let $W$ be a subsemigroup of $S$. Let $R$ be a commutative unitary ring.
If $d(S,R)$ is finite, then $d(W,R)\leq d(S,R)$.
\end{lemma}

\begin{lemma}\label{Lemma Greens class} (see \cite{Grillet monograph}, Chapter I, Proposition 4.3) \ Let $S$ be a commutative semigroup, and let $s$ be an element of $S$. Then $H_s$ is a group if and only if $H_s$ contains an idempotent.
\end{lemma}

\begin{lemma}\label{Lemma Schuzebure} (see  \cite{Grillet monograph}, Chapter I, Proposition 4.6) \ Let $S$ be a commutative semigroup, and let $s$ be an element of $S$ such that ${\rm St}(H_s)\neq \emptyset$. Then $\Gamma(H_s)$ is a simply transitive group of permutations of $H_s$, and $t\mapsto \gamma_t$ is a homomorphism of ${\rm St}(H_s)$ onto $\Gamma(H_s)$. If $H_s$ is a group, then $\Gamma(H_s)\cong H_s$.
\end{lemma}

\begin{lemma}\label{Lemma maximal subgroup} (see \cite{Grillet monograph}, Chapter I, Corollary 4.5) \ Let $S$ be a commutative semigroup. The maximal subgroups of $S$ coincide with the Green's classes which contains idempotents. They are pairwise disjoint. Every subgroup of $S$ is contained in exactly one maximal subgroups.
\end{lemma}

\begin{lemma} \label{Lemma relative and DAvenport} (see Proposition 2.1 of \cite{wangAddtiveirreducible}) \ Let $S$ be a commutative semigroup. Then ${\rm D(S)}$ is finite if, and only if ${\rm D}_a(S)$ is bounded for all $a\in S$, i.e., there exists
a given large integer $M$ such that ${\rm D}_a(S)\leq M$ for all $a\in S$. In particular, if ${\rm D}(S)$ is
finite, then ${\rm D}(S)=1+\max\limits_{a\in S} \{{\rm D}_a(S)\}$.
\end{lemma}

\begin{lemma} \label{Lemma relative and quotient} \ Let $S$ be a finite commutative semigroup.
For any $a\in S$, we have ${\rm D}_{\bar{a}}(S\diagup \mathcal{H})\leq \psi(a)+1$, where $\bar{a}$ denotes the image of $a$ under the canonical epimorphism of $S$ onto $S\diagup \mathcal{H}$.
\end{lemma}

\begin{proof}  Denote $\ell={\rm D}_{\bar{a}}(S\diagup \mathcal{H})$. Take $a_1,\ldots,a_{\ell}\in S$ such that $\bar{a}_1\cdot\ldots\cdot\bar{a}_{\ell}$ is an additive irreducible sequence over the quotient semigroup $S\diagup \mathcal{H}$ with $\sum\limits_{i=1}^{\ell}\bar{a}_i=\bar{a}$, i.e.,
$\sum\limits_{i=1}^{\ell}a_{i}  \ \mathcal{H} \ a$.  By the irreducibility of of the sequence $\bar{a}_1\cdot\ldots\cdot\bar{a}_{\ell}$, we have that
$a_1\succ_{\mathcal{H}}(a_1+a_2)\succ_{\mathcal{H}}\cdots \succ_{\mathcal{H}} \sum\limits_{i=1}^{\ell}a_{i} \ \mathcal{H} \ a,$ which implies that $\ell-1\leq \psi(a)$, done.
\end{proof}

\begin{lemma}\label{Lemma splitting field still} Let $S$ be a finite commutative semigroup with $\exp(S)=n$, and let $K$ be a splitting field of $S$. For any $a\in S$ such that ${\rm St}(H_a)\neq \emptyset$, then $\exp(\Gamma(H_a))\mid n$ and
$K$ is also a splitting field of the group $\Gamma(H_a)$.
\end{lemma}

\begin{proof} We need only to prove $\exp(\Gamma(H_a))\mid n$. Take an arbitrary element $b$ of ${\rm St}(H_a)$.
Then
$\langle b \rangle\subset {\rm St}(H_a)$.
 Let $e=\rho(b) b$ be the idempotent in the cyclic semigroup $\langle a\rangle$. By Lemma \ref{Lemma Schuzebure}, we have that $\gamma_e=\gamma_{e+e}=\gamma_{e}+\gamma_{e}$ and so $\gamma_{e}$ is the identity element $0_{\Gamma}$ of the group $\Gamma(H_a)$.
 $\mathcal{P}(b)\gamma_b=\mathcal{P}(b)\gamma_b+0_{\Gamma}=\mathcal{P}(b)\gamma_b+\gamma_e
 =\gamma_{\mathcal{P}(b)b}+\gamma_{\rho(b)b}=\gamma_{\mathcal{P}(b)b+\rho(b)b}=\gamma_{\rho(b)b}=\gamma_e=0_{\Gamma}$, which implies that \begin{equation}\label{equation the order bdivides pb}
 o(\gamma_b)\mid \mathcal{P}(b).
 \end{equation}
 By \ref{Lemma Schuzebure}, the homomorphism $x\mapsto \gamma_x: {\rm St}(H_a)\rightarrow \Gamma(H_a)$ is surjective. Combined the arbitrariness of $b$ and \eqref{equation the order bdivides pb}, we have that $\exp(\Gamma(H_a))={\rm lcm}\{o(\gamma_x): x\in {\rm St}(H_a)\}\mid {\rm lcm} \{\mathcal{P}(b): b\in {\rm St}(H_a)\}\mid \exp(S)$, we are done.
\end{proof}

\begin{lemma}\label{Lemma group act on V} Let $S$ be a commutative semigroup, and let $s$ be an element of $S$. Suppose that ${\rm St}(H_s)\neq \emptyset$. Let $V=\{x\in S: x\preceq_{\mathcal{H}} s\}.$ Then the group $\Gamma(H_s)$ act on the set $V$ given by:
$\gamma_a\circ x=a+x$ for any $a\in {\rm St}(H_s)$ and $x\in V$.
\end{lemma}

\begin{proof} Take any $x\in V$. Then $x=s+b$ for some $b\in S$. Take an element $c\in {\rm St}(H_s)$ such that $\gamma_c$ is the identity element of the group $\Gamma(H_s)$. We see that $\gamma_c\circ x=c+x=c+(s+b)=(c+s)+b=(\gamma_c\circ s)+b=s+b=x$. Moreover, take any two elements $\alpha,\beta\in {\rm St}(H_s)$. We see that $(\gamma_{\alpha}+\gamma_{\beta}) \circ x=(\gamma_{\alpha+\beta}) \circ x=(\alpha+\beta)+x=\alpha+(\beta+x)=\gamma_{\alpha}\circ (\gamma_{\beta} \circ x)$.  Then the conclusion follows from the action of groups on a set readily.
\end{proof}

\begin{lemma}\label{Lemma elments of V} Let $S$ be a finite commutative semigroup, and let $g_1,\ldots,g_t$ be elements of $S$ (not necessarily distinct). For any subset $I$ of $[1,t]$, we have $\sum\limits_{i\in I} g_i+\sum\limits_{j\in [1,t]\setminus I} e(g_j)\preceq_{\mathcal{H}} \sum\limits_{i=1}^t g_i$.
\end{lemma}

\begin{proof} We see that $\sum\limits_{i\in I} g_i+\sum\limits_{j\in [1,t]\setminus I} e(g_j)=\sum\limits_{i\in I} g_i+\sum\limits_{j\in [1,t]\setminus I} (e(g_j)+e(g_j))=\sum\limits_{i\in I} g_i+\sum\limits_{j\in [1,t]\setminus I}2\rho(g_j) g_j=\sum\limits_{i=1}^t g_i+\sum\limits_{j\in [1,t]\setminus I}(2\rho(g_j)-1) g_j\preceq_{\mathcal{H}} \sum\limits_{i=1}^t g_i$, done.
\end{proof}

\begin{lemma}\label{Lemma all in He} Let $S$ be a finite commutative semigroup, and let $g_1,\ldots,g_{t}$ be elements of $S$ (not necessarily distinct) such that $\langle \sum\limits_{i=1}^{t} g_i\rangle$ is a group. Then $H_e$ is a group with $e$ being the identity element of $H_e$, where $e=\sum\limits_{i=1}^{t} e(g_i)$. Moreover,
$\sum\limits_{i\in I} g_i+\sum\limits_{j\in [1,t]\setminus I} e(g_j)\in H_e$ for each subset $I$ of $[1,t]$.
\end{lemma}

\begin{proof} By Lemma \ref{Lemma Greens class}, we have that $H_e$ is a group. Since $e$ is an idempotent of $H_e$ and a group has a unique idempotent, it follows that $e$ is the identity element of the group $H_e$. By Lemma \ref{lemma with idempotent}, we see $e=\sum\limits_{i=1}^{t} e(g_i)=e(\sum\limits_{i=1}^{t}g_i)=\rho(\sum\limits_{i=1}^{t}g_i)\sum\limits_{i=1}^{t}g_i\in \langle \sum\limits_{i=1}^{t} g_i\rangle$. Since $\langle \sum\limits_{i=1}^{t} g_i\rangle$ is a group, it follows that $e \mathcal{H} \sum\limits_{i=1}^{t}g_i$. By applying Lemma \ref{Lemma elments of V}, we derive that  $\sum\limits_{i=1}^{t}g_i\succeq _{\mathcal{H}} \sum\limits_{i\in I} g_i+\sum\limits_{j\in [1,t]\setminus I}e(g_j)\succeq _{\mathcal{H}}\sum\limits_{i\in I} e(g_i)+\sum\limits_{j\in [1,t]\setminus I}e(g_j)=\sum\limits_{i=1}^{t}e(g_i)=e$. This implies that $\sum\limits_{i\in I} g_i+\sum\limits_{j\in [1,t]\setminus I}e(g_j) \mathcal{H} e$, i.e., $\sum\limits_{i\in I} g_i+\sum\limits_{j\in [1,t]\setminus I}e(g_j)\in H_e$. \end{proof}

\begin{lemma}\label{Lemma sum of gi} Let $S$ be a finite commutative semigroup, and let $g_1,\ldots,g_{\ell}$ be elements of $S$ (not necessarily distinct).
For any subset $I$ of $[1,t]$, we have $\sum\limits_{i\in I} g_i+\sum\limits_{j\in [1,t]\setminus I} e(g_j)\preceq_{\mathcal{H}} \sum\limits_{i=1}^t g_i$.
\end{lemma}

\begin{lemma}\label{Lemma homomorphism semigrouringtogroupring}  Let $S$ be a commutative semigroup, and let $R$ be a commutative unitary ring.  Let $a\in S$ such that ${\rm St}(H_a)\neq \emptyset$.  Then the semigroup homomorphism $\pi_a:{\rm St}(H_a)\rightarrow \Gamma(H_a)$ given in \eqref{equation homo to shuzen}
can be extended to an epimorphism, denoted $\widetilde{\pi}_a$, of the semigroup ring $R[X;{\rm St}(H_a)]$ onto the group ring $R[X;\Gamma(H_a)]$ as  $\widetilde{\pi}_a: \sum\limits_{g\in G} c_g X^g\mapsto \sum\limits_{g\in G} c_g X^{\gamma_g}$
where $\sum\limits_{g\in G} c_g X^g\in K[X;{\rm St}(H_b)]$. Moreover, for any $g\in {\rm St}(H_a)$, then $(X^{g}-X^{e(g)})\in R[X;{\rm St}(H_a)]$ and $\widetilde{\pi}_a((X^{g}-X^{e(g)}))=(X^{\gamma_g}-X^{0_{\Gamma(H_a)}})\in R[X;\Gamma(H_b)]$.
\end{lemma}

\begin{proof}  By \eqref{equation operation} and \eqref{equation homo to shuzen}, we can check that $\widetilde{\pi}_s$ is an epimorphism of  $R[X;{\rm St}(H_a)]$ onto $R[X;\Gamma(H_a)]$. Moreover, since $R$ is unitary ring, the $\widetilde{\pi}_s$ is an extension of $\pi_s$. Since ${\rm St}(H_a)$ is a subsemigroup and $g\in {\rm St}(H_a)$, it follows that $e(g)=\rho(g)g\in {\rm St}(H_a)$. Moreover, since $\gamma_{e(g)}+\gamma_{e(g)}=\gamma_{e(g)+e(g)}=\gamma_{e(g)}$, it follows that $\gamma_{e(g)}=0_{\Gamma(H_a)}$, then the conclusion follows readily.
\end{proof}

\begin{lemma}\label{Lemma archimedean semigroup at most} [Grillet, P72. Proposition 1.3] An archimedean semigroup contains at most one idempotent.
\end{lemma}

The following lemma are known in commutative semigroup theory without citations directly. For the reader convenience, we give a short proof below.

\begin{lemma}\label{Lemma f.c.s is archi iff one idemptoent} Let $S$ be a finite commutative semigroup. Then $S$ is archimedean if and only if $S$ has a unique idempotent.
\end{lemma}

\begin{proof} We first show the sufficiency. Denote $e$ to be the unique idempotent of $S$. Take any two elements $a,b$ of $S$. We see that $\rho(a)a=e=e+e=2\rho(b)b=b+(2\rho(b)-1)b$, and thus, $S$ is Archimedean.
Hence, it remains to prove the necessity.
Since a finite semigroup contains at least one idempotent, then the conclusion follows from Lemma \ref{Lemma archimedean semigroup at most}.
\end{proof}

\begin{lemma}\label{Lemma clifford semigroup} Let $S$ be a commutative semigroup.  Then $S$ is a clifford semigroup if and only if every element of $S$ is contained in some subgroup of $S$.
\end{lemma}

\begin{proof} The sufficiency follows trivially. To prove the necessity, we take an element $a$ of $S$. By Lemma \ref{Lemma Greens class}, it suffices to prove that $a\in H_e$ for some idempotent.
By definition of commutative clifford semigroups,  there exists $a'\in S$ such that $aa'a=a$ and $a'aa'=a'$. Then $(aa')(aa')=(aa'a)a'=aa'$ and so $aa'$ is an idempotent of $S$. Denote $e=aa'$. Then $e\preceq_{\mathcal{H}} a$. On the other hand, $a=a(a'a)=ae$, and $a\preceq_{\mathcal{H}} e$. Hence, $a\mathcal{H} e$ and $a\in H_e$, done.
\end{proof}

\noindent {\bf Proof of Theorem \ref{Theorem General bound for finite commutative semigroup}.} \
 Since ${\rm D}(S\diagup\mathcal{H})<\infty$, it follows from
Lemma \ref{Lemma relative and DAvenport} that $${\rm D}_{\bar{a}}(S\diagup\mathcal{H})<\infty \mbox{ for each } a\in S.$$
Suppose $d(\Gamma(H_a),R)<\infty$ for all $a\in S$.
Let
\begin{equation}\label{equation ell=1+max}
\ell=1+\max\limits_{a\in S}\left\{\epsilon_a \ {\rm D}_{\bar{a}}(S\diagup\mathcal{H})+d(\Gamma(H_a),K)\right\}.
 \end{equation}
 Let $T=s_1\cdot\ldots \cdot s_{\ell}$ be an arbitrary sequence over $S$ of length $\ell$.

\noindent \textbf{Claim A.} There exist $c_1,\ldots,c_{\ell}\in K^{\times}$ such that $\prod\limits_{i=1}^{\ell}(X^{s_i}-c_i X^{e(s_i)})=0$.

{\sl Proof of Claim A}.
Denote $$s=\sum\limits_{i=1}^{\ell} s_i.$$
 We first consider the case that $\langle s\rangle$ is a not group, i.e.,
\begin{equation}\label{equation epsilon s=1}
\epsilon_s=1.
\end{equation}
By the definition of ${\rm D}_{\bar{s}}(S\diagup\mathcal{H})$, we see that there exists exist distinct integers $i_1,i_2,\ldots,i_t$ of $[1,\ell]$ with
\begin{equation}\label{equation t leq D}
t\leq {\rm D}_{\bar{s}}(S\diagup\mathcal{H})
\end{equation}
 such that $(\sum\limits_{j=1}^t  \bar{s}_{i_j})=\bar{s}$, i.e., $(\sum\limits_{j=1}^t  s_{i_j})\ \mathcal{H} \ s$. Say $\{i_1,i_2,\ldots,i_t\}=\{1,2,\ldots,t\}$. Then
\begin{equation}\label{equation sum to t H s}
(\sum\limits_{j=1}^t  s_j)\ \mathcal{H} \ s.
 \end{equation}
It follows from \eqref{equation ell=1+max},  \eqref{equation epsilon s=1} and  \eqref{equation t leq D} that
\begin{equation}\label{equation ell-t>0}
\begin{array}{llll}
\ell-t &\geq& 1+\max\limits_{a\in S}\left\{\epsilon_a \ {\rm D}_{\bar{a}}(S\diagup\mathcal{H})+d(\Gamma(H_a),K)\right\}-{\rm D}_{\bar{s}}(S\diagup\mathcal{H})\\
&\geq & 1+\epsilon_s \ {\rm D}_{\bar{s}}(S\diagup\mathcal{H})+d(\Gamma(H_s),K)-{\rm D}_{\bar{s}}(S\diagup\mathcal{H})\\
&=& 1+d(\Gamma(H_a),K)\\
&>&0.\\
\end{array}
\end{equation}
For any $i\in [t+1,\ell]$, since $s \preceq_{\mathcal{H}} s_i+(\sum\limits_{j=1}^t  s_j) \preceq_{\mathcal{H}} \sum\limits_{j=1}^t  s_j \mathcal{H} s$, it follows that \eqref{equation sum to t H s} that
$(s_i+ s) \mathcal{H} s.$
which implies that
 \begin{equation}\label{equation si+s H s}
s_j\in {\rm St}(H_s)
\end{equation}
 and so $$\Gamma(H_s)\neq \emptyset.$$
Denote $N={\rm St}(H_s)$ and $\Gamma=\Gamma(H_s).$
 Let $R[X;N]$ be the semigroup ring of the semigroup $N$ over $R$, and let $R[X;\Gamma]$ be the group ring of the group $\Gamma$ over $R$. By \eqref{equation operation} and \eqref{equation homo to shuzen}, we can check that $\pi_s$ can be extended to an ring-homomorphism, denoted $\widetilde{\pi}_s$, of $R[X;N]$ to $R[X;\Gamma]$, given as
 \begin{equation}\label{equation tilde pi}
 \widetilde{\pi}_s: \sum\limits_{g\in G} a_g X^g\mapsto \sum\limits_{g\in G} a_g X^{\gamma_g}
  \end{equation}
  where $\sum\limits_{g\in G} a_g X^g\in R[X;N]$.
 Let
  \begin{equation}\label{equation V}
 V=\{b\in S: b\preceq_{\mathcal{H}} s\}.
   \end{equation}
 Then $V$ is a nonempty subsemigroup of $S$.
 We see that
 \begin{equation}\label{equation a+b in V}
 a+b\in V \mbox{ for any } a\in  N \mbox{ and any } b\in V.
 \end{equation}
  Let $R[X;V]$ be the semigroup ring of the semigroup $V$ over $R$. We shall view $R[X;V]$ as a free $R$-module in the canonical way with a basis $\{X^{v}: v\in V\}$. By \eqref{equation a+b in V}, we can show that the following map $$\varphi: R[X;N] \rightarrow {\rm End}_R(R[X;V])$$ is a representation of the ring $R[X;N]$ on the module $R[X;V]$:
  \begin{equation}\label{equation represen1}
 \varphi(\sum\limits_{g\in N} a_g X^g): \sum\limits_{v\in V} b_v X^v\mapsto (\sum\limits_{g\in N} a_g X^g)(\sum\limits_{v\in V} b_v  X^v)=\sum\limits_{g\in N}\sum\limits_{v\in V}a_g b_v X^{g+v}\in R[X;V],
 \end{equation}
  where $\sum\limits_{g\in N} a_g X^g\in R[X;N]$
and  $\sum\limits_{v\in V} b_v  X^v\in R[X;V]$. By Lemma \ref{Lemma group act on V}, the group $\Gamma$ acts on the set $V$. Hence, we can show the following map $$\psi:K[X;\Gamma] \rightarrow {\rm End}_R(R[X;V])$$ is a representation
 of the ring $R[X;\Gamma]$ on the the module $R[X;V]$:
   \begin{equation}\label{equation represen2}
 \psi(\sum\limits_{\gamma\in \Gamma} a_{\gamma} X^{\gamma}):\sum\limits_{v\in V} b_vX^v\mapsto \sum\limits_{\gamma\in \Gamma}\sum\limits_{v\in V} a_{\gamma} b_v X^{\gamma\circ v},
   \end{equation}
  where $\sum\limits_{\gamma\in \Gamma} a_{\gamma} X^{\gamma}\in R[X;\Gamma]$
and $\sum\limits_{v\in V} b_v  X^v\in R[X;V]$.
By \eqref{equation gammac}, \eqref{equation tilde pi}, \eqref{equation represen1} and \eqref{equation represen2}, we can show that
\begin{equation}\label{equation composite}
\varphi=\psi\widetilde{\pi}_s,
\end{equation} i.e., the representation $\varphi$ is a lift of the representation $\psi$ by the ring homomorphism $\widetilde{\pi}_s$. Recall \eqref{equation ell-t>0},
$\ell-t\geq 1+d(\Gamma(H_s),R)$, it follows from the definition of $d(\Gamma(H_s),R)$ that there exist
 $c_{t+1},\ldots,c_{\ell}\in K^{\times}$ such that $$\prod\limits_{i=t+1}^{\ell}(X^{\gamma_{s_i}}-c_i X^{0_\Gamma})=0\in R[X;\Gamma].$$ By \eqref{equation sum to t H s} and Lemma \ref{Lemma elments of V}, we derive that
$(\prod\limits_{j=1}^{t}(X^{s_j}- X^{e(s_j)})\in R[X;V]$. By \eqref{equation represen1}, \eqref{equation represen2} and \eqref{equation composite}, we conclude that
$$\begin{array}{llll}
& &\prod\limits_{i=t+1}^{\ell}(X^{s_i}-c_i X^{e(s_i)})
\prod\limits_{j=1}^{t}(X^{s_j}- X^{e(s_j)})\\
 &=& \varphi(\prod\limits_{i=t+1}^{\ell}(X^{s_i}-c_i X^{e(s_i)}))(\prod\limits_{j=1}^{t}(X^{s_j}- X^{e(s_j)}))\\
&=& \psi\widetilde{\pi}_s(\prod\limits_{i=t+1}^{\ell}(X^{s_i}-c_i X^{e(s_i)}))(\prod\limits_{j=1}^{t}(X^{s_j}- X^{0_\Gamma}))\\
&=& \psi(\prod\limits_{i=t+1}^{\ell}(X^{\gamma_{s_i}}-c_i X^{0_\Gamma}))(\prod\limits_{j=1}^{t}(X^{s_j}- X^{e(s_j)}))\\
&=&0_{K[X;\Gamma]}(\prod\limits_{j=1}^{t}(X^{s_j}- X^{e(s_j)}))=0_{\in R[X;V]}=0_{R[X;S]}.\\
\end{array}$$

Now we consider the case that $\langle s\rangle$ is a group, i.e.,  $\epsilon_s=0$.
Denote $e=\sum\limits_{i=1}^{\ell} e(s_i)$. Let $s_i'=e+s_i$ for each $i\in [1,\ell]$.
Then $s_i'\preceq_{\mathcal{H}} e$ where $i\in [1,\ell]$. We see that
$e=e+e=e+\sum\limits_{i=1}^{\ell} e(s_i)=(e+s_i)+(2\rho(s_i)-1)s_i+\sum\limits_{j\in [1,\ell]\setminus \{i\}}e(s_j) \preceq_{\mathcal{H}} e+s_i=s_i'$, which implies
\begin{equation}\label{equation here si'inHe}
s_i'\in H_e \mbox{ for each } i\in [1,\ell].
\end{equation} By Lemma \ref{Lemma Schuzebure}, we have $d(\Gamma(H_s),K)=d(H_s,K)$.
Since $\ell=1+\max\limits_{a\in S}\left\{\epsilon_a \ {\rm D}_{\bar{a}}(S\diagup\mathcal{H})+d(\Gamma(H_a),K)\right\}\geq 1+d(\Gamma(H_s),K)=1+d(H_s,K)$, it follows from \eqref{equation here si'inHe}  and Lemma \ref{Lemma all in He} that
that there exist $c_1,\ldots,c_{\ell}\in R^{\times}$ such that $(X^{s_1'}-c_1X^{e})\cdot\ldots\cdot (X^{s_{\ell}'}-c_{\ell}X^{e})=0$, and thus,
$$\begin{array}{llll}
& &(X^{s_1}-c_1X^{e(s_1)})\cdots(X^{s_{\ell}}-c_{\ell}X^{e(s_\ell)})\\
&=&\sum\limits_{I\subsetneq [1,\ell]}(-1)^{|I|}(\prod\limits_{i\in I} c_i)
X^{\sum\limits_{j\in [1,\ell]\setminus I}s_j+\sum\limits_{i\in I} e(s_i)}+(-1)^{\ell}(\prod\limits_{i=1}^{\ell} c_i) X^{\sum\limits_{i=1}^{\ell} e(s_i)}\\
&=&\sum\limits_{I\subsetneq [1,\ell]}(-1)^{|I|}(\prod\limits_{i\in I} c_i)
X^{e+\sum\limits_{j\in [1,\ell]\setminus I}s_j+\sum\limits_{i\in I} e(s_i)}+(-1)^{\ell}(\prod\limits_{i=1}^{\ell} c_i) X^{e}\\
&=&\sum\limits_{I\subsetneq [1,\ell]}(-1)^{|I|}(\prod\limits_{i\in I} c_i)
X^{e+\sum\limits_{j\in [1,\ell]\setminus I}s_j}+(-1)^{\ell}(\prod\limits_{i=1}^{\ell} c_i) X^{e}\\
&=&\sum\limits_{I\subsetneq [1,\ell]}(-1)^{|I|}(\prod\limits_{i\in I} c_i)
X^{\sum\limits_{j\in [1,\ell]\setminus I}(e+s_j)}+(-1)^{\ell}(\prod\limits_{i=1}^{\ell} c_i) X^{e}\\
&=&\sum\limits_{I\subsetneq [1,\ell]}(-1)^{|I|}(\prod\limits_{i\in I} c_i)
X^{\sum\limits_{j\in [1,\ell]\setminus I}s_j'}+(-1)^{\ell}(\prod\limits_{i=1}^{\ell} c_i) X^{e}\\
&=&(X^{s_1'}-c_1X^{e})\cdot\ldots\cdot (X^{s_{\ell}'}-c_{\ell}X^{e})=0.\\
\end{array}$$
Therefore, we proves Claim A. \qed

By the arbitrariness of choosing the sequence $T$ and Claim A, we prove that if $d(\Gamma(H_a),R)<\infty$ for each $a\in S$, then $d(S,R)<\infty$ and $d(S,R)\leq\max\limits_{a\in S}\left\{\epsilon_a \ {\rm D_{\bar{a}}}(S\diagup \mathcal{H}) +d(\Gamma(H_a),R)\right\}$.

It remains to show that $d(S,R)<\infty$ implies that $d(\Gamma(H_a),R)<\infty$ for each $a\in S$ and that $d(S,R)\geq \max\limits_{a\in S}\left\{d(\Gamma(H_a),R)\right\}$. We take an arbitrary element $b\in S$ such that
$\Gamma(H_b)\neq \emptyset$, i.e., ${\rm St}(H_b)\neq \emptyset$, and take an arbitrary sequence $\gamma_1\cdot\gamma_2\cdot\ldots\cdot\gamma_m$ over the group $\Gamma(H_b)$ which is algebraically zero-sum free sequence (over $R$), i.e.,
$\prod\limits_{i=1}^m(X^{\gamma_i}-c_i X^{0_{\Gamma}})\neq 0\in R[X;\Gamma(H_b)]$ for all $c_1,\ldots,c_{m}\in R\setminus \{0\}$.
Since $\pi_b$ is an epimorphism of ${\rm St}(H_b)$
onto $\Gamma(H_b)$, it follows that there exist $g_1,g_2,\ldots,g_m\in {\rm St}(H_a)$ such that $\pi_b(g_i)=\gamma_i$ for each $i\in [1,m]$. Then by applying Lemma \ref{Lemma homomorphism semigrouringtogroupring},
we have that $\widetilde{\pi}_b(\prod\limits_{i=1}^m (X^{g_i}-c_i X_{e(g_i)}))=\prod\limits_{i=1}^m (X^{\gamma_i}-c_i X_{0})\neq 0$ for all $c_1,\ldots,c_{m}\in R\setminus \{0\}$, and so $\prod\limits_{i=1}^m (X^{g_i}-c_i X_{e(g_i)}\neq 0$. This implies that $d(S,R)\geq m$. By the arbitrariness of choosing $b$ and choosing the algebraically zero-sum free sequence $\gamma_1\cdot\gamma_2\cdot\ldots\cdot\gamma_m$, we have the conclusion proved. The lower bound and the upper bound will be attained which can be seen in Theorem \ref{Theorem for finite}, Corollary \ref{Corollary Clifford} and Corollary \ref{Corollary Archimedean}. \qed

\noindent {\bf Proof of Theorem \ref{Theorem for finite}.}  As noted in the remark, ${\rm D}(S\diagup \mathcal{H})<\infty$. For each $a\in S$ with $\Gamma(H_a)\neq \emptyset$, by Lemma \ref{Lemma splitting field still}, we see that Condition (i) implies that $K$ is a splitting field of $\Gamma(H_a)$, and Condition (ii) implies that $\exp(\Gamma(H_a))$ is also a $p$-power. Then it follows from Theorem C (ii) and (iii) that $d(\Gamma(H_a),K)<\infty$ for any $a \in S$. By Theorem \ref{Theorem General bound for finite commutative semigroup}, we have that $d(S,K)<\infty$ and $$\max\limits_{a\in S}\left\{d(\Gamma(H_a),K)\right\} \leq d(S,K)\leq\max\limits_{a\in S}\left\{\epsilon_a \ {\rm D_{\bar{a}}}(S\diagup \mathcal{H}) +d(\Gamma(H_a),K)\right\}.$$ By Lemma \ref{Lemma relative and quotient}, we have ${\rm D_{\bar{a}}}(S\diagup \mathcal{H})\leq \psi(a)+1$ for any $a\in S$, and the theorem is proved. \qed

\noindent {\bf Proof of Corollary \ref{Corollary Clifford}.} By Lemma \ref{Lemma clifford semigroup}, we see $\langle a\rangle$ is a subgroup of $S$
and so
 $\epsilon_a=0$. By Theorem \ref{Theorem for finite}, we have that $d(S,K)<\infty$ and $d(S,K)=\max\limits_{a\in S}\left\{d(\Gamma(H_a),K)\right\}$.  Let $e$ be the identity of the group $\langle a\rangle$. Since $a \mathcal{H} e$, it follows from Lemma \ref{Lemma Greens class} that $H_a=H_e$ is a subgroup of $S$.   By Lemma \ref{Lemma Schuzebure}, $\Gamma(H_a)\cong H_a$. Combined with Lemma \ref{Lemma maximal subgroup}, we have the conclusion proved. \qed

\noindent {\bf Proof of Corollary \ref{Corollary Archimedean}.}

(i). We first show that
\begin{equation}\label{equation one point}
\max\limits_{a\in S}\left\{\epsilon_a \ (\psi(a)+1)+d(\Gamma(H_a),K)\right\}=\max\left(\Psi(S)+1,d(G_S,K)\right).
\end{equation}
By Lemma \ref{Lemma f.c.s is archi iff one idemptoent}, we denote $e$ to be the unique idempotent of $S$.  Note that
\begin{equation}\label{equation e leq every a}
e\preceq_{\mathcal{H}} a \mbox{ for every } a\in S.
\end{equation}
 This is because that $e=ne=a+t$ for some $n>0$ and $t\in S$ by the definition for archimedean semigroups.
 It follows from \eqref{equation e leq every a} that
 \begin{equation}\label{equation max psi a=psi S-1}
 \max\limits_{x\in S\setminus H_e}\{\psi(x)\}=\Psi(S)-1.
 \end{equation}
 By Lemma \ref{Lemma Greens class}, we see that $H_e$ is the largest subgroup of $S$ consisting of all element $a$ such that $\langle a\rangle$ is a group. That is,
  \begin{equation}\label{equation epsa=0or1}
 \begin{array}{llll}\epsilon_a=\left\{\begin{array}{llll}
               0,  & \mbox{if \ \ } {a\in H_e};\\
               1,  &  otherwise.\
              \end{array}
              \right.
\end{array}
\end{equation}
By \eqref{equation max psi a=psi S-1}, we conclude that  \begin{equation}\label{equation He=e+x}
H_e=\{e+x: x\in S\}.
\end{equation}
Then we have the following.

\noindent \textbf{Claim B}. If $b\in S\setminus H_e$, then ${\rm St}(H_b)=\emptyset$ and so $d(\Gamma(H_a),K)=0$.

{Proof of Claim B.} To prove ${\rm St}(H_b)=\emptyset$, we suppose to the contrary that there exists some $b\in S\setminus H_e$ such that ${\rm St}(H_b)\neq \emptyset$.
Take $d\in {\rm St}(H_b)$. Then $e+b=\rho(d)d+b\in H_b$, which implies $e+b\notin H_e$, a contradiction with \eqref{equation He=e+x}. This proves that ${\rm St}(H_b)=\emptyset$. It follows that $\Gamma(H_a)=\emptyset$ and so $d(\Gamma(H_a),K)=0$. \qed

By \eqref{equation max psi a=psi S-1}, \eqref{equation epsa=0or1},  Claim A and Lemma \ref{Lemma Schuzebure}, we derive that
\begin{equation}\label{equation one point}\begin{array}{llll}
& &\max\limits_{a\in S}\left\{\epsilon_a \ (\psi(a)+1)+d(\Gamma(H_a),K)\right\}\\
&=&\max\left( \max\limits_{a\in S\setminus H_e}\{\epsilon_a \ (\psi(a)+1)+d(\Gamma(H_a),K)\},\max\limits_{a\in H_e}\{\epsilon_a \ (\psi(a)+1)+d(\Gamma(H_a),K)\}\right)
\\
&=&\max\left(\max\limits_{a\in S\setminus H_e}\{\psi(a)+1\},\max\limits_{a\in H_e}\{d(\Gamma(H_a),K)\}\right)\\
&=&\max\left( \Psi(S), d(H_e,K)\right)\\
\end{array}
\end{equation}
By Theorem \ref{Theorem General bound for finite commutative semigroup}, we see that $d(S,K)\leq\max\limits_{a\in S}\left\{\epsilon_a \ (\psi(a)+1)+d(\Gamma(H_a),K)\right\}$. By \eqref{equation one point}, to prove Conclusion (i), it suffices to show that $d(S,K)\geq\max\left( \Psi(S), d(H_e, K)\right)$.

If $d(H_e,K)\geq \Psi(S)$, since $H_e$ is a subgroup of $S$, it follows from Lemma \ref{Lemma subsemigroup} that $d(S;K)\geq d(H_e,K)=\max\left( \Psi(S), d(H_e,K)\right)$, and we are done.  Hence, we assume $$\Psi(S)>d(H_e,K).$$ By \eqref{equation e leq every a} and the definition for $\Psi(S)$, we can take elements $g_1,\ldots,g_{t}\in S$ such that
$$g_1\succ_{\mathcal{H}}(g_1+g_2)\succ_{\mathcal{H}}\cdots \succ_{\mathcal{H}} \sum\limits_{i=1}^{t}g_{i} \succ_{\mathcal{H}} e,$$ where $t=\Psi(S)$.
Take arbitrary $a_1,\ldots,a_t\in K\setminus \{0\}$. Recall that $e(g_i)=e$ since $e$ is the unique idempotent of $S$.
Let $f=\prod\limits_{i=1}^{t}(X^{g_i}-a_i X^e).$
 We see that $f=\prod\limits_{i=1}^{t}(X^{g_i}-a_i X^e)=X^{\sum\limits_{i=1}^t g_i}+\sum\limits_{\emptyset\neq J\subsetneq [1,t]}(-1)^{t-|J|}
(\prod\limits_{j\in [1,t]\setminus J}a_j)X^{e+\sum\limits_{j\in J}g_j}+(-1)^t(\prod\limits_{i=1}^t a_i)X^e$. Since $\sum\limits_{i=1}^t g_i\notin H_e$, it follows from \eqref{equation He=e+x} that $f\neq 0$.
By the arbitrariness of $a_1,\ldots,a_t$ taking from $K\setminus \{0\}$, we have that $d(S,K)\geq t=\Psi(S)=\max\left( \Psi(S), d(H_e, K)\right)$, completing the proof of Conclusion (i).

(ii) This conclusion follows from Conclusion (i) and Theorem C immediately. \qed

\noindent {\bf Proof of Corollary \ref{Corollary elementarysemigroup}}. Let $S=G\cup N$ where $G$ is a group and $N$ is a nilsemigroup and $G$ acts on $N$.
By Theorem \ref{Theorem for finite}, it suffices to show that $d(S,K)\geq  \epsilon_s (\psi(s)+1)+d(\Gamma(H_s),K)$ for any $s\in S$.

Suppose $s\in G$. Then $\epsilon_s=0$ and $\Gamma(H_s)\cong G$. It follows from Lemma \ref{Lemma subsemigroup} that
$d(S,K)\geq  d(G,K)= \epsilon_s \ (\psi(s)+1)+d(\Gamma(H_s),K)$, done.

Suppose $s=\infty$ is the zero element of the nilsemigroup $N$, which is  also the zero element of the semigroup $S$.
We see that $\epsilon_s=0$, $|H_s|=1$ and the $|\Gamma(H_s)|=1$ which implies that $d(\Gamma(H_s),K)=0$. Then $d(S,K)
\geq \epsilon_s \ (\psi(s)+1)+d(\Gamma(H_s),K)$ holds trivially.

Hence, we need only to consider the case that $$s\in N\setminus \{\infty\}.$$
Let $T=s_1\cdot\ldots\cdot s_t$ such that $s_1\succ_{\mathcal{H}}(s_1+s_2)\succ_{\mathcal{H}}\cdots \succ_{\mathcal{H}} \sum\limits_{i=1}^{t}s_{i} \mathcal{H} s,$ where $t=\psi(a)+1$. Since $G$ acts on the nilsemigroup, we have that $s_i\in N$ for each $i\in [1,t]$. Then we take a sequence $U=g_1\cdot\ldots\cdot g_m$ over $G$ such that the sequence
\begin{equation}\label{equation all for ai neq 0}
\prod\limits_{i=1}^m (X^{\gamma_{g_i}}-a_i X^{0_{\Gamma}})\neq 0\in K[X;G] \mbox{ for all } a_1,\ldots,a_{m}\in K\setminus \{0\}.
\end{equation}
with $m=d(\Gamma(H_s),K)$.

Then we take arbitrary $b_1,\ldots,b_{m+t}\in K\setminus \{0\}$. Since $e(s_j)=\infty$ for each $j\in [1,t]$ and
Then $\prod\limits_{j=1}^{m} (X^{g_j}-b_{t+j} X^{e(s_j)})\prod\limits_{i}^{t} (X^{s_i}-b_i X^{e(s_i)})=
\prod\limits_{j=1}^{m} (X^{\gamma_{g_j}}-b_{t+j} X^{0_{\Gamma}})\prod\limits_{i}^{t} (X^{s_i}-b_i X^{\infty})=
\prod\limits_{j=1}^{m} (X^{\gamma_{g_j}}-b_{t+j} X^{0_{\Gamma}})(X^{\sum\limits_{i=1}^t s_i}+ c X^{\infty})=
\prod\limits_{j=1}^{m} (X^{\gamma_{g_j}}-b_{t+j} X^{0_{\Gamma}})X^{\sum\limits_{i=1}^t s_i}+\prod\limits_{j=1}^{m} (X^{\gamma_{g_j}}-b_{t+j} X^{0_{\Gamma}})c X^{\infty}=\prod\limits_{j=1}^{m} (X^{\gamma_{g_j}}-b_{t+j} X^{0_{\Gamma}})X^{\sum\limits_{i=1}^t s_i}+d X^{\infty}$ where $c,d\in K$.  By lemma, we see that the action of $\Gamma$ on $H_s$ is simply.
Since $\prod\limits_{j=1}^{m} (X^{\gamma_{g_j}}-b_{t+j} X^{0_{\Gamma}})\neq 0$, it follows that $\prod\limits_{j=1}^{m} (X^{\gamma_{g_j}}-b_{t+j} X^{0_{\Gamma}})X^{\sum\limits_{i=1}^t s_i}\in K[X;H_s]\setminus \{0\}$. Since $d X^{\infty}\notin K[X;H_s]$, it follows that $\prod\limits_{j=1}^{m} (X^{g_j}-b_{t+j} X^{e(s_j)})\prod\limits_{i}^{t} (X^{s_i}-b_i X^{e(s_i)})\neq 0$, completing the proof of the conclusion.
\qed

Note that the direct product of cyclic semigroups is an archimedean semigroup. Hence, as application of Corollary \ref{Corollary Archimedean}, we have the following.

\begin{prop}\label{Proposition in the product cyclic semigroups} For $r\geq 1$, let $S={\rm C}_{k_1; n_1}\times \cdots \times {\rm C}_{k_r; n_r}$ be a direct product of cyclic semigroups with $k_1\leq k_2\leq \cdots\leq k_r$. Let $K$ be a field. Let $G_S=\bigoplus\limits_{i=1}^r \mathbb{Z}_{n_i}$. Suppose that either (i) $K$ is an algebraically closed field of characteristic zero, or (ii) ${\rm Char}(K)=p$ and $\prod\limits_{i=1}^{r} n_i$ is a $p-$power.
Then $d(S,K)=\max\left(k_r-1,d(\bigoplus\limits_{i=1}^r \mathbb{Z}_{n_i};K)\right)\geq \max\left(k_r-1,{\rm D}(\bigoplus\limits_{i=1}^r \mathbb{Z}_{n_i})-1\right),$ and the equality $d(S,K)= \max\left(k_r-1,{\rm D}(\bigoplus\limits_{i=1}^r \mathbb{Z}_{n_i})-1\right)$ holds if $r=1$, or ${\rm Char}(K)=p$.
\end{prop}

\begin{proof} By Lemma \ref{Lemma cyclic semigroup}, we can derive that $S$ contains a unique idempotent (one can also see Lemma 3.3 in \cite{wangCommunication}), and moreover, $\Psi(S)=k_r-1$. By Lemma \ref{Lemma f.c.s is archi iff one idemptoent}, we have that $S$ is archimedean. Observe that $G_s\simeq \bigoplus\limits_{i=1}^r \mathbb{Z}_{n_i}$. Then the conclusion follows from Corollary \ref{Corollary Archimedean},  Theorem C.
\end{proof}

In the end of this section, we shall give a semigroup, in which $d(S,R)<\infty$ but ${\rm D}(S\diagup H)=\infty$.

\begin{exam}\label{Example D not finite} Let $R$ be unitary commutative ring. Let $S=\mathbb{N}_{>1}$ be the set of positive integers. The operation on $S$ is given by: $ab\triangleq {\rm lcm}(a,b)$ for any $a,b\in N$. It is easy to check that $S$ is a commutative semilattice, and so a commutative periodic semigroup. Moreover, we can check that $d(S,R)=0$ since every element is idempotent, and that ${\rm D}(S\diagup H)=\infty$, since any $n$ pairwise coprime integers greater than $1$ forms an additive irreducible sequence over the semigroup $S$.
\end{exam}

Example \ref{Example D not finite} shows that the condition ${\rm D}(S\diagup H)<\infty$ is not necessary such that $d(S,R)<\infty$ for a commutative periodic semigroup $S$. Below we give a theorem by replacing the condition ${\rm D}(S\diagup H)<\infty$ with a new one. Moreover,  we shall give
a conjecture (see Conjecture \ref{conjecture 1}) on the necessary and sufficient conditions such that $d(S,R)<\infty$.

\begin{theorem}\label{Theorem rivising} Let $R$ be a commutative unitary ring, and
let $S$ be a commutative periodic semigroup such that ${\rm sup}\{{\rm D_{\bar{a}}}(S\diagup \mathcal{H}: a\in S \mbox{ with } \epsilon_a=1\}<\infty$. Then $d(S,R)<\infty$ if and only if $d(\Gamma(H_a),R)<\infty$ for each $a\in S$. Moreover, in case of  $d(S,R)<\infty$, we have
the following best possible upper and lower bounds $$\max\limits_{a\in S}\left\{d(\Gamma(H_a),R)\right\} \leq d(S,R)\leq\max\limits_{a\in S}\left\{\epsilon_a \ {\rm D_{\bar{a}}}(S\diagup \mathcal{H}) +d(\Gamma(H_a),R)\right\},$$ where $\epsilon_a$ and $\bar{a}$ are given as Theorem \ref{Theorem General bound for finite commutative semigroup}.
\end{theorem}

\begin{proof} Using the same arguments of Theorem \ref{Theorem General bound for finite commutative semigroup}.
\end{proof}

Then we close this section with the following conjecture.

 \begin{conj}\label{conjecture 1} Let $S$ be a commutative periodic semigroup, and let $R$ be a commutative unitary ring. If $d(S,R)<\infty$ then ${\rm sup}\left\{{\rm D}_{\bar{a}}(S\diagup \mathcal{H}): a\in S \mbox{ and } \langle a\rangle \mbox{  is not a group} \right\}<\infty$.
\end{conj}

\begin{remark} If Conjecture \ref{conjecture 1} is verified,  by Theorem \ref{Theorem rivising} we shall have the following conclusion immediately: For any commutative periodic group $S$ and any unitary commutative ring $R$,  then $d(S;R)<\infty$ if, and only if, $${\rm sup}\{{\rm D}_{\bar{a}}(S\diagup \mathcal{H}): a\in S \mbox{ and }\langle a\rangle \mbox{ is not a group} \}<\infty$$ and $d(\Gamma(H_s),R)<\infty$ for each $s\in S$.
\end{remark}

\section{d(S;R) with $R$ being of a prime characteristic}

In this section, we shall discuss the connection of the invariant $d(S;R)$ with the generators of the semigroups $S$ in case that $R$ has a prime characteristic. To begin with, we need to introduce  some notation and lemmas.  It is well known that any finitely generated periodic commutative semigroup is finite. Hence, we shall focus on the finite commutative semigroup below.  For any finite commutative semigroup $S$, we define $$\Lambda(S)=\min\limits_A\{1+\sum\limits_{a\in A} (\rho(a)-1)\}$$ where $A$ takes all generating sets of the semigroup $S$.
Suppose that $G$ is a finite abelian group. Then $\Lambda(G)\leq d^*(G)+1,$  and equality holds if and only if $G$ is a finite abelian $p$-group. The main theorem in this section is the following theorem.

\begin{theorem}\label{Theorem main theorem} \  Let $p$ be a prime.  Let $S$ be a finite commutative semigroup such that $\exp(S)$ is a $p$-power, and let $R$ be commutative unitary ring of characteristic $p$. Then the following conclusions hold.

(i) Suppose $T=s_1\cdot\ldots\cdot s_{\ell}$ is a sequence over $S$ such that $\ell \geq \Lambda(S)$. Then $\prod\limits_{i=1}^{\ell} (X^{s_i}-X^{e(s_i)})=0\in R[X;S]$, and moreover, ${\rm St}(e)\cap \sum(T)\neq \emptyset$ where $e=\sum\limits_{i=1}^m e(s_i)$;

(ii) $d(S,R)\leq \Lambda(S)-1.$
\end{theorem}

To Prove Theorem \ref{Theorem main theorem}, we shall need some lemmas as follow.
\begin{lemma}\label{lemma factorization} \ Let $S$ be a finite commutative semigroup and $R$ a
commutative unitary ring $R$. Let $a_1,\ldots,a_k\in S$ (not necessarily distinct), and $s=\sum\limits_{i=1}^k a_i$. Then $X^{s}-X^{e(s)}=\sum\limits_{i=1}^k \beta_i (X^{a_i}-X^{e(a_i)})$  for some   $\beta_1,\ldots, \beta_k\in R[S].$
\end{lemma}

\begin{proof} Say $k=2$. By Lemma \ref{lemma with idempotent}, we have $X^{s}-X^{e(s)}=X^{a_1+a_2}-X^{e(a_1+a_2)}=X^{a_1}X^{a_2}-X^{e(a_1)}X^{e(a_2)}=\beta_1 (X^{a_1}-X^{e(a_1)})+\beta_2 (X^{a_2}-X^{e(a_2)})$ where $\beta_1=X^{a_2}\in R[S]$ and $\beta_2=X^{e(a_1)}\in R[S]$, done. Then the lemma follows readily by the induction on $k$.
\end{proof}

\begin{lemma}\label{lemma monomial of power rho(b)} \ Let $S$ be a commutative semigroup, and let $p$ be a prime number. Let $R$ be commutative unitary ring of characteristic $p$. For any periodic element $b$ of $S$ such that the period $\mathcal{P}(b)$ is a $p$-power, we have $(X^{b}-X^{e(b)})^{\rho(b)}=0\in R[S]$.
\end{lemma}

\begin{proof} Say $\mathcal{P}(b)=p^t$ for some $t\geq 0$. By \eqref{equation rho(x)=knn}, we see $\mathcal{P}(b)\mid \rho(b)$.  Then we set $\rho(b)=n p^t$ where $n\geq 1$. Denote $e=e(b)$.
By Lemma \ref{Lemma cyclic semigroup}, we see that
\begin{equation}\label{equation kP(b)+e=e}
k\mathcal{P}(b)b+e=k\mathcal{P}(b)b+\rho(b)b=e \mbox{ for every } k>0.
\end{equation}
Since $R$ is commutative and of characteristic $p$, it follows from \eqref{equation kP(b)+e=e} that $(X^{b}-X^{e})^{\rho(b)}= [(X^{b}-X^{e})^{p^t} ]^n=[(X^{b})^{p^t}-(X^{e})^{p^t}]^n=[X^{p^t b}-X^{p^t e}]^n=[X^{p^t b}-X^{e}]^n=(X^{p^t b})^n+(-1)^n (X^{e})^n +\sum\limits_{i=1}^{n-1} {n \choose i}(-1)^{n-i} (X^{p^t b})^i \cdot (X^e)^{n-i}=X^{np^t b}+(-1)^n X^{ne}+\sum\limits_{i=1}^{n-1} {n \choose i}(-1)^{n-i} X^{ip^t b+(n-i)e}=X^e+(-1)^n X^e+\sum\limits_{i=1}^{n-1} {n \choose i}(-1)^{n-i} X^{ip^t b+e}=X^e+(-1)^n X^e+\sum\limits_{i=1}^{n-1} {n \choose i}(-1)^{n-i} X^e=(X^e-X^e)^n =0\in R[S]$.
\end{proof}

\begin{lemma}\label{lemma if equals zeor} \ Let $S$ be a commutative periodic semigroup.  Let $R$ be commutative unitary ring. Let $T=s_1\cdot\ldots\cdot s_{m}$ be a sequence over $S$.
If $\prod\limits_{i=1}^m (X^{s_i}-X^{e(s_i)})=0\in R[S]$ then ${\rm St}(e)\cap \sum(T)\neq \emptyset$, where $e=\sum\limits_{i=1}^m e(s_i)$.
\end{lemma}

\begin{proof}   Denote $e_i=e(s_i)$ for $i\in [1,m]$.  By expanding out $\prod\limits_{i=1}^m (X^{s_i}-X^{e(s_i)})$, we conclude that (i) $\sum\limits_{j\in [1,m]} s_j=e$, or (ii) there exists a nonempty subset $J\subsetneq [1,m]$ such that $\sum\limits_{j\in J} s_j+ \sum\limits_{i\in [1,m]\setminus J} e_i=\sum\limits_{i\in [1,m]} e_i=e$. For case (i), we have $\sum\limits_{j\in [1,m]} s_j=e\in {\rm St}(e)$. Hence, we assume (ii) holds.
Then  $\sum\limits_{j\in J} s_j+e=\sum\limits_{j\in J} s_j+\sum\limits_{i\in [1,m]} e_i=\sum\limits_{j\in J} s_j+ (\sum\limits_{i\in [1,m]\setminus J} e_i+\sum\limits_{i\in [1,m]} e_i)=(\sum\limits_{j\in J} s_j+ \sum\limits_{i\in [1,m]\setminus J} e_i)+\sum\limits_{i\in [1,m]} e_i=e+e=e$, and so $\sum\limits_{j\in J} s_j\in  {\rm St}(e)\cap\sum(T)$, we are done.
\end{proof}

Now we are in a position to prove Theorem \ref{Theorem main theorem}.

\noindent {\bf Proof of Theorem \ref{Theorem main theorem}} \ (i). Let $A$ be a generating set of $S$ such that
$1+\sum\limits_{a\in A} (\rho(a)-1)=\Lambda(S)$. For each $i\in [1,\ell]$, since $s_i$ can be represented as a finite sum of elements (not necessarily distinct) from $A$, it follows from Lemma \ref{lemma with idempotent} and Lemma \ref{lemma factorization} that
$X^{s_i}- X^{e(s_i)}=\sum\limits_{a\in A} \beta_{i,\ a} (X^{a}-X^{e(a)})$ where $\beta_{i,\ a}\in R[X;S]$. Thus, $\prod\limits_{i=1}^{\ell} (X^{s_i}- X^{e(s_i)})$ can be represented as a sum of the following monomials:
\begin{equation}\label{equation monomial 1}
\beta \cdot \prod\limits_{a\in A}  (X^{a}-X^{e(a)})^{t_a}
\end{equation}
 where $\beta \in R[X;S]$, $t_a\geq 0$ and $\sum\limits_{a\in A} t_a=\ell$. Since $\ell\geq \Lambda(S)=1+\sum\limits_{a\in A} (\rho(a)-1)$, it follows that there exists some $b\in A$ such that $t_b\geq \rho(b)$. By Lemma \ref{lemma monomial of power rho(b)}, we derive that the monomial given as \eqref{equation monomial 1} satisfies $\beta \cdot \prod\limits_{a\in A}  (X^{a}-X^{e(a)})^{t_a}=\beta \cdot (X^{b}-X^{e(b)})^{\rho(b)}\cdot (X^{b}-X^{e(b)})^{t_b-\rho(b)}\prod\limits_{a\in A\setminus \{b\}}  (X^{a}-X^{e(a)})^{t_a}=0$.  Then ${\rm St}(e)\cap \sum(T)\neq \emptyset$ follows from Lemma \ref{lemma if equals zeor}.

(ii) This conclusion follows readily from Definition \ref{Definition d(S,R)} and Conclusion (i). \qed

\begin{prop}\label{prop using} Let $S$ be a commutative periodic semigroup, and let $T=s_1\cdot\ldots\cdot s_{m}$ be a sequence over $S$. For any domain $R$, and $a_1,\ldots, a_m\in R\setminus \{0_R\}$. Let $$f=\prod\limits_{i=1}^m(X^{s_i}-a_iX^{e(s_i)}) =\sum\limits_{s\in S} c_s X^{s}$$
with $c_s\in R$ for every $s\in S$. If ${\rm St}(e)\cap\sum(T)=\emptyset$ then $c_e\neq 0_R$ where $e=\sum\limits_{i=1}^m e(s_i)$, and in particular, $f\neq 0$.
\end{prop}

Proof. Using the same arguments as in Lemma \ref{lemma if equals zeor}.

In the section of concluding remarks, we shall give one clifford semigroup (Propositions \ref{Prop Clifford semigroup}) in which the condition ${\rm St}(e)\subset E(S)$ for every $e\in E(S)$ is satisfied. Then ${\rm St}(e)\cap\sum(T)=\emptyset$ if and only if $T$ is idempotent-sum free.

\section{Concluding remarks}

We first remark that if the semigroup $S$ meets some conditions,  a combinatorial explanation on the connection between $d(S,R)$ and the Erd\H{o}s-Burgess constant $I(S)$ can be obtained, which are given in the following Propositions \ref{proposition I(S)leq d(S,R)+1} and \ref{Prop Clifford semigroup}.

\begin{prop}\label{proposition I(S)leq d(S,R)+1} Let $S$ be a commutative semigroup, and let $R$ be a commutative unitary ring.
Suppose that ${\rm St}(e)\subset E(S)$ for every $e\in E(S)$. Then $I(S)\leq d(S,R)+1$.
\end{prop}

\begin{proof} Let $T=s_1\cdot\ldots\cdot s_{\ell}$ be a sequence over $S$ with $\ell=d(S,R)+1$. By the definition $d(S,R)$, we see that $\prod\limits_{i=1}^{\ell} (X^{s_i}-X^{e(s_i)})=0$.
 It follows from Lemma \ref{lemma if equals zeor} that ${\rm St}(e)\cap \sum(T)\neq \emptyset$ and so $\sum(S)\cap E(S)\neq \emptyset$, done.
\end{proof}

\begin{prop}\label{Prop Clifford semigroup} \ Let $S$ be a finite commutative Clifford semigroup, and let $R$ be a commutative unitary ring. Let $\mathbb{G}=(G,\gamma)$ be the group functor on the universal semilattice $Y(S)$ associated with the semigroup $S$. Suppose that the assigned homomorphism $\gamma_y^x$ is injective for every pair $(x,y)$ of elements in $Y(S)$ such that $y\preceq x$. Then the following conclusions hold:

(i) $I(S)\leq {\rm d}(S,R)+1$;

(ii) Suppose $R$ is a field with ${\rm Char}(R)=p$ and $\exp(S)$ is a $p$-power for some prime $p$. Then $I(S)={\rm d}(S,R)+1$.
\end{prop}

\begin{proof} (i). We first show the following.

\noindent \textbf{Claim C.} For any $e\in E(S)$, we have ${\rm St}(e)\subset E(S)$.

\noindent {\sl Proof of Claim C}. Suppose that $g\in {\rm St}(e)$, i.e., $g+e=e.$ Let $x,y\in Y(S)$ be such that $g\in G_x$ and $e\in G_y$. We see that $y\preceq x$ in the semilattice. It follows that
$e=g+e=\gamma_{y}^{x}(g)+\gamma_{x}^{x}(e)=\gamma_{y}^{x}(g)+e$, and so $\gamma_{y}^{x}(g)=e$ is the identity element in the group $G_e$. Since $\gamma_{y}^{x}$ is a monomorphism, it follows that $g=e(g)\in E(S)$ is the identity element in the group $G_x$. This proves Claim C. \qed

Suppose to the contrary that $I(S)-1> {\rm d}(S,R)$. We take an idempotent-sum free sequence $T=g_1\cdot\ldots\cdot g_{\ell}$ of length $\ell=I(S)-1$ over $S$. Then $$(X^{g_1}-a_1X^{e_1})\cdot\ldots\cdot (X^{g_{\ell}}-a_{\ell} X^{e_{\ell}})=0\in R[S]$$ for some $a_1,\ldots,a_{\ell}\in R\setminus \{0\}$. By Proposition \ref{prop using}, we derive that ${\rm St}(e)\cap \sum(T)\neq \emptyset$. Combined with Claim C, we have that $E(S)\cap \sum(T)\neq \emptyset$, which contradicts with $T$ being idempotent-sum free.  This proves (i).

(ii) By Corollary \ref{Corollary Clifford}, we have $d(S,R)=d(H,R)$ for some subgroup $H$ of $S$. Since $H$ is a $p$-group, it follows from Theorem C and Theorem \ref{Corollary Clifford} that $I(H)={\rm D}(H)=d(H,R)+1$, done.
\end{proof}

 As given in Corollary \ref{Corollary Archimedean} and \ref{Corollary Clifford}, the upper bounds are attained for the classes of clifford semigroups and archimedean semigroups, which represents the two different kind of semigroups structures. Hence propose the following two conjectures.

\begin{conj}\label{Conjcture 2} Let $S$ be a finite commutative semigroup, and let $K$ be a field.
Suppose that either (i) $K$ is a splitting field of $S$, or (ii) ${\rm Char}(K)=p$ and $\exp(S)$ is a $p$-power. Then
$$d(S,K)=\max\limits_{a\in S}\left\{\epsilon_a \ {\rm D_{\bar{a}}}(S\diagup \mathcal{H}) +d(\Gamma(H_a),K)\right\},$$ where
$\begin{array}{llll}\epsilon_a=\left\{\begin{array}{llll}
               0,  & \mbox{if \ \ } {\langle a\rangle} \mbox{ is a group };\\
               1,  &  otherwise.\
              \end{array}
              \right.
\end{array}$
\end{conj}

\begin{remark} We remark that although the bound $d(S,K)\leq \max\limits_{a\in S}\left\{\epsilon_a \ (\psi(a)+1)+d(\Gamma(H_a),K)\right\}$ is attained in Clifford semigroups, Archimedean semigroup and elementary semigroups (which cover all types of irreducible components in both semilattice decomposition and subdirectly product decomposition of a finite commutative semmigroups), there exists some finite commutative semigroups $S$ for which
$$d(S,K)<\max\limits_{a\in S}\left\{\epsilon_a(\psi(a)+1)+d(\Gamma(H_a),K)\right\}<\infty$$ (see Example \ref{exam 1}).  That is the reason why in Conjecture \ref{Conjcture 2},
the conjectured equality $d(S,K)=\max\limits_{a\in S}\left\{\epsilon_a \ {\rm D_{\bar{a}}}(S\diagup \mathcal{H}) +d(\Gamma(H_a),K)\right\}$ can not be replaced by $d(S,K)=\max\limits_{a\in S}\left\{\epsilon_a(\psi(a)+1)+d(\Gamma(H_a),K)\right\}$.
\end{remark}

\begin{exam} \label{exam 1} Take $t$ ($t\geq 2$) finite cyclic nilsemigroups generated by $x_1,\ldots,x_t$ respectively with $|\langle x_i\rangle|=n>1$ for each $i\in [1,t]$. Let  $K$ be an algebraically closed field with ${\rm Char}(K)=0$.
Let $S=\bigcup \limits_{i=1}^t \langle x_i \rangle$ be the disjoint union of the $m$ cyclic nilsemigroups.
We define the operation on $S$, denoted $+_{S}$, given by
$$\begin{array}{llll}a+_{S}b=\left\{\begin{array}{llll}
               a+_{\langle x_i \rangle} b,  & \mbox{if \ \ }  i=j;\\
               b,  &  \mbox{if \ \ }  i<j;\\
               a,  &  \mbox{if \ \ }  i>j,
              \end{array}
              \right.
\end{array}$$
for any $a,b\in S$ such that $a\in  \langle x_i \rangle$ and $b\in  \langle x_j\rangle$, where $+_{\langle x_i\rangle}$ denotes the addition in the cyclic semigroup $\langle x_i\rangle$.
It is easy to check that $S$ is a finite commutative semigroup, with $H_s=\{s\}$ and so $d(\Gamma(H_s),K)=0$, for each $s\in S$. Then we have that  $$x_1\succ_{\mathcal{H}} 2x_1\succ_{\mathcal{H}}\cdots\succ_{\mathcal{H}}n x_1 \succ_{\mathcal{H}} x_2\succ_{\mathcal{H}}2x_2 \succ_{\mathcal{H}}\cdots  \succ_{\mathcal{H}} nx_2 \ \ \ \succ_{\mathcal{H}} \ \ \ \cdots \succ_{\mathcal{H}} x_t\succ_{\mathcal{H}} 2x_t\succ_{\mathcal{H}}\cdots\succ_{\mathcal{H}}n x_t$$
Since $nx_t$ is an idempotent, then $\max\limits_{a\in S}\left\{\epsilon_a(\psi(a)+1)+d(\Gamma(H_a),K)\right\}=\max\limits_{a\in S\setminus E(S)}\left\{\epsilon_a(\psi(a)+1)\right\}=\psi((n-1)x_t)+1=tn-1$. On the other hand, we see that ${\rm D}_{\bar{a}}(S\diagup \mathcal{H})\leq n$. Therefore, we have that $\max\limits_{a\in S}\left\{\epsilon_a \ {\rm D_{\bar{a}}}(S\diagup \mathcal{H}) +d(\Gamma(H_a),K)\right\}\leq n<tn-1=\max\limits_{a\in S}\left\{\epsilon_a(\psi(a)+1)+d(\Gamma(H_a),K)\right\}$.
\end{exam}

\bigskip

\noindent {\bf Acknowledgements}

\noindent
This work is supported by NSFC (grant no. 12371335, 12271520).

\end{document}